\documentclass[%
 aip,
 amsmath,amssymb,
 reprint,%
jmp]{revtex4-2}

\usepackage{graphicx}
\usepackage{dcolumn}
\usepackage{bm}
\usepackage[mathlines]{lineno}

\usepackage{amsthm}
\usepackage{amsmath}
\usepackage{amssymb}
\usepackage{mathrsfs}%
\usepackage{mathtools}
\usepackage{mathptmx}
\usepackage{etoolbox}

\makeatletter
\def\@email#1#2{%
 \endgroup
 \patchcmd{\titleblock@produce}
  {\frontmatter@RRAPformat}
  {\frontmatter@RRAPformat{\produce@RRAP{*#1\href{mailto:#2}{#2}}}\frontmatter@RRAPformat}
  {}{}
}%
\makeatother

\newtheorem{theorem}{Theorem}[section]
\newtheorem*{definition}{Definition}
\newtheorem{corollary}{Corollary}[theorem]
\newtheorem{proposition}{Proposition}[theorem]
\newtheorem{lemma}{Lemma}[theorem]

\begin{document}

\preprint{AIP/123-QED}

\title[A General Theory of Operator-Valued Measures]{A General Theory of Operator-Valued Measures}
\author{{Luis A.} {Cede\~no-P\'erez}$^1$ and Hernando Quevedo$^{1,2,3}$}

\email{luisacp@ciencias.unam.mx,quevedo@nucleares.unam.mx}

\affiliation{ Instituto de Ciencias Nucleares, Universidad Nacional Aut\'onoma de M\'exico, AP  70543, Mexico City, Mexico }
\affiliation{ Dipartimento di Fisica and Icra, Universit\`a di Roma “La Sapienza”, Roma, Italy }
\affiliation{ Al-Farabi Kazakh National University, Al-Farabi av. 71, 050040 Almaty, Kazakhstan}

\date{\today}

\begin{abstract}
We construct a new kind of measures, called projection families, that generalize the classical notion of vector- and operator-valued measures. The maximal class of reasonable functions admits an integral with respect to a projection family, where the integral is defined as an element of the second dual instead of the original space. We show that projection families possess strong enough properties to satisfy the theorems of Monotone Convergence and Dominated Convergence, but are much easier to come by than the more restrictive operator-valued measures. 
\end{abstract}


\maketitle

\tableofcontents

\section{Introduction}

\subsection{Vector Measures}

Vector measures have played an important role in mathematics and physics since the twentieth century. They were first considered in the general version of the spectral theorem for operators between Hilbert spaces. This result provides a function $E$ that maps measurable sets of the spectrum into projections, with respect to which complex functions can be integrated to obtain an operator. The function $E$ is such that
\begin{equation*}
    T = \int_{\sigma(T)}\lambda\;dE(\lambda),
\end{equation*}
which additionally allows us to define the evaluation of integrable functions on the operator $T$ as the operator that results from integrating the function with respect to $E$, that is,
\begin{equation*}
    f(T) = \int_{\sigma(T)}f\;dE.
\end{equation*}
In this case, $E$ is known as a \textbf{spectral measure} or \textbf{resolution of the identity} associated to the operator $T$. The development of the theory of quantum information has also considered two similar types of operator measures (see \cite{Davies}). The first is the case of measures $\mathcal{P}$ that assign a positive operator $\mathcal{P}(A)$ to measurable sets $A$ and the second is the case of measures $\mathcal{E}$ that to each measurable set $A$ assign a bounded linear operator $\mathcal{E}(A)$ on the space of states of a Hilbert space $H$, also known as the space of density operators in $H$. In both cases, the measures define an integral
\begin{equation*}
    \int f\;d\mathcal{E}
\end{equation*}
which in turn is an operator of the same kind as the measure $\mathcal{E}$.

These three theories of integration are developed independently from each other, but their common characteristics make it natural to wonder if the three of them are particular cases of a more general theory of operator-valued measures. The theories of spectral measures and positive-operator-valued-measures (POVM) are similar since both measures have values in the space of bounded operators on a Hilbert space (see \cite{McLaren} for an approach in Hilbert spaces); however, measures that act over the space of states, known as operation-valued measures, are measures that act not on a Hilbert space but on a Banach space, which alters the theory considerably. This difference is such that the theory of integration with respect to operation-valued measures can not be a copy with minor differences of the same theory for spectral measures.

\subsection{Incomplete History of Vector Integration}

A first step towards a general theory for operator-valued measures is the similar problem studied in Non-linear Analysis of Integration in Banach spaces (\cite{Diestel}, \cite{Dinculeanu} and \cite{Graves}). The basic problems are to define the integral of a scalar-valued function $f$ with respect to a vector measure $\mu$ and of a vector function $F$ with respect to a scalar measure $\lambda$. The first solution to these problems consists in first defining the integral of a simple function as
\begin{equation*}
    \int f\;d\mu = \sum_{n=1}^{m}a_{n}\mu(E_{n})
\end{equation*}
and
\begin{equation*}
    \int F\;d\lambda = \sum_{n=1}^{m}x_{n}\lambda(E_{n}),
\end{equation*}
respectively. In the usual Lebesgue theory, one would proceed to take suprema over the integrals of functions of this kind. This is not possible in either case since the previous sums are vectors in a Banach space and not real numbers. For this reason, this process has to be replaced by a limit of integrals of simple functions to define the integral of more general functions. The resulting integrals are known as the Dunford-Schwartz and Bochner integrals, respectively. Even though these two integrals are properly defined, they are rather clumsy since the limit process complicates their calculation and the class of functions that can be integrated is rather small, since these are functions that can be adequately approximated by simple functions. The key step in the development of these integrals was to abandon the idea of approximation by simple functions in favor of a definition that depends on the behavior of the integral with respect to elements of the dual space. Given an element of the dual space $\Lambda$, it was proven that the integrals of Dunford-Schwarz and Bochner satisfy the relations
\begin{equation*}
    \Lambda\left(\int f\;d\mu\right) = \int f\;d\Lambda\circ\mu
\end{equation*}
and
\begin{equation*}
    \Lambda\left(\int F\;d\lambda\right) = \int \Lambda\circ F\;d\lambda,
\end{equation*}
respectively. These formulas uniquely determine the integral in Banach spaces, which implies that the integral can be defined as the only element of the space that satisfies them. These integrals are known as the Lewis and Pettis integrals, respectively (see \cite{Lewis} and \cite{Pettis}). This way of defining the integral has the additional advantage of being easily generalized to topological vector spaces whose dual separates points. The main difficulty of these integrals is the question of existence, as their computation is rather straightforward when compared to the Dunford-Schwarz and Bochner integrals. The difficulty of existence lies in the fact that the definition does not provide a way to approximate the integral. In the Banach space setting, this is solved by allowing the integral to be an element of the second dual space $X^{\ast\ast}$ instead of the original space $X$. Since $X\subset X^{\ast\ast}$, this solution consists in allowing the integral to exist in a space larger than the one usually considered. Furthermore, in the Banach space setting, it was proved that every reasonable function, in an adequate sense, admits an integral in $X^{\ast\ast}$. These are known as the generalized Lewis integral and the Dunford integral, respectively.

There is another integral that allows the integration of vector functions $F$ with respect to vector measures $\mu$, known as Bartle's bilinear integral (see \cite{Bartle}). This integral additionally requires a bounded bilinear form $T\colon X\times X \to X$, usually denoted as a product, such that the integral of a simple function is given by
\begin{equation*}
    \int F\;d\mu = \sum_{n=1}^{m}T(x_{m},\mu(E_{m})).
\end{equation*}
Bartle then defined the integral of more general functions by means of a limit of integrals of simple functions. Despite the generality and good properties of Bartle's integral, this integral has the same limitations as the Dunford-Schwartz and Bochner, which originate from defining the integral as a limit. The solution to this would be to find an expression for the Bartle integral that only makes reference to the dual spaces of the spaces involved. To our knowledge, this has not been achieved.

Most of the results of vector integration have fallen into obscurity because of their complexity and lack of literature. The existence of the Lewis integral seems to be unknown to researchers in mathematics.

\subsection{Projection Families}

The case of operator-valued measures seems suitable for applying the Lewis theory of vector measures, considered as measures on the Banach space $B(X)$ of linear bounded operators in $X$. In Lewis' theory, the dual space plays a fundamental role, which is problematic since the dual space of $B(X)$ is rarely known. To confront this problem, the theory we develop applies a pointwise procedure that depends on the dual of $X$ instead of that of $B(X)$. Furthermore, we prescind the idea of the measure assigning an operator to every measurable set in favor of a family of measures that defines the projections of an operator. If to every pair $x\in X$ and $\Lambda\in X^{\ast}$ we assign a measure $\mu_{\Lambda,x}$, to be thought as a function of the form $A\longmapsto \mu_{\Lambda,x}(A)$, then there may exist an operator $\int f\;d\mu$ such that
\begin{equation*}
    \Lambda\left(\int f\;d\mu(x)\right) = \int f\;d\mu_{\Lambda,x}.
\end{equation*}
If the family of measures $\mu_{\Lambda,x}$ varies continuously with respect to $\Lambda$, in an adequate sense, we expect to be able to reconstruct the operator $\int f\;d\mu(x)$ from the given projections. Likewise, if the family of measures continuously changes in $x$, in an adequate sense, we expect an operator that to each $x\in X$ assigns the element $\int f\;d\mu(x)$ to define a linear bounded operator. In this work, we prove that all previous assertions can be made true, thus generalizing the integration theory of Lewis and verifying that the previous theories of integration with respect to operator-valued measures are indeed particular cases of this theory. As a final application, we prove a generalization of the spectral theorem for operators between Banach spaces. (This construction is far too involved and will be presented in a follow-up article).

\subsection{Notation}

Throughout this work $X$ will be a Banach space, $X^{\ast}$ its dual space, $X^{\ast\ast}$ its second dual space and $J$ will be the canonical injection of $X$ into $X^{\ast\ast}$ that maps $x$ to the element $J(x)$ defined as
\begin{equation*}
    J(x)(\Lambda) = \Lambda(x).
\end{equation*}
$B_{X}$ and $B_{X^{\ast}}$ will denote the closed unit balls centered in zero of $X$ and $X^{\ast}$, respectively.

The topologies $\tau_{\omega}$, $\tau_{\omega^{\ast}}$ and $\tau_{b\omega^{\ast}}$ will denote the weak, weak-$\ast$ and bounded weak-$\ast$ topologies, respectively. The convergence with respect to each of these topologies will be denoted by the corresponding subindex, for example, $x_{i}\xrightarrow[]{\omega}x$, $\Lambda_{i} \xrightarrow[]{\omega^{\ast}} \Lambda$ and $\Lambda_{i} \xrightarrow[]{b\omega^{\ast}} \Lambda$, respectively. We recall that $J$ is a homeomorphism into its image from $\tau_{\omega}$ into $\tau_{\omega^{\ast}}$, $X$ is reflexive if $J$ is surjective and a functional from $X^{\ast}$ the field is continuous with respect to $\tau_{\omega^{\ast}}$ if and only if it is continuous with respect to $\tau_{b\omega^{\ast}}$ and both are equivalent to being in the image of $J$.

The pair $(\Omega,\Sigma)$ will denote a measurable space. We will say that a net of measures $(\mu_{i})_{i\in I}$ converges setwise if $\mu_{i}(E) \to \mu(E)$ for every $E\in\Sigma$, which we will denote by $\mu_{i}\xrightarrow[]{set} \mu$. Convergence almost everywhere will be denoted as $f_{i} \xrightarrow[]{a.e.} f$ and pointwise convergence as $f_{i}\xrightarrow[]{pw} f$. Similarly, the convergence in $L^{p}(\mu)$ will be denoted as $f_{i} \xrightarrow[]{L^{p}} f$. If $\mu$ is a complex measure, $|\mu|$ will denote its total variation measure and $|\mu|_{TV}$ its total variation norm. $\mathcal{M}(\Omega)$ will denote the set of complex measures in $(\Omega,\Sigma)$.

Even if $X$ is a Banach, and perhaps not Hilbert, space, we will denote by $\tau_{SO}$ the strong operator topology of $B(X)$, that is, the topology generated by the seminorms
\begin{equation*}
    p_{x}(T) = |T(x)|
\end{equation*}
associated to each $x\in X$.

\section{Vector Measures}

\subsection{Vector Projection Families}

We begin with the case of measures that take values in a Banach space $X$.

\begin{definition}
Let $X$ be a Banach space, and $(\Omega,\Sigma)$ a measurable space. A \textbf{vector projection family} is a collection of measures in $\Omega$, denoted by
\begin{equation*}
    \mu = \{\mu_{\Lambda}\;|\;\Lambda\in X^{\ast}\},
\end{equation*}
with the following properties.
\begin{enumerate}
    \item The function
    \begin{equation*}
    \begin{array}{ccc}
        X^{\ast} & \longrightarrow &\mathcal{M}(\Omega) \\
        \Lambda & \longmapsto & \mu_{\Lambda}
    \end{array}
    \end{equation*}
    defines a linear functional.
    \item If $(\Lambda_{i})_{i\in I}$ is a net in $X^{\ast}$ such that $\Lambda_{i} \to \Lambda$ then
    \begin{equation*}
        \mu_{\Lambda_{i}} \xrightarrow[]{set} \mu_{\Lambda}.
    \end{equation*}
\end{enumerate}
\end{definition}

\begin{definition}
Let $X$ be a Banach space, $(\Omega,\Sigma)$ a measurable space, and $\mu$ a vector projection family.
\begin{enumerate}
    \item We will say that $E\in\Sigma$ is a \textbf{null set} if $\mu_{\Lambda}(E) = 0$ for each $\Lambda\in X^{\ast}$. In this case, we will write $E\in \mathcal{N}(\mu)$.

    \item We will say that a measurable $f\colon \Omega \to \mathbb{C}$ is \textbf{essentially bounded} with respect to $\mu$ if there exists a null set $E$ such that $f$ is bounded in $\Omega\setminus E$. In this case, we will write $f\in L^{\infty}(\mu)$.

    \item We will say that a function $f\colon \Omega \to \mathbb{C}$ is \textbf{integrable} if $f\in L^{1}(\mu_{\Lambda})$ for each $\Lambda\in X^{\ast}$. In this case, we will write $f\in L^{1}(\mu)$.
\end{enumerate}
\end{definition}

Note that since each $\mu_{\Lambda}$ is finite, we also have that $L^{\infty}(\mu) \subset L^{1}(\mu)$.

\begin{definition}
Given a vector projection family $\mu$ and $f\in L^{1}(\mu)$ we define the \textbf{integral} of $f$ \textbf{with respect to} $\mu$ as the aplication
\begin{equation*}
\begin{array}{cccc}
    \int_{\Omega}f\;d\mu\colon& X^{\ast} & \longrightarrow &\mathbb{C} \\
    & \Lambda & \longmapsto & \int_{\Omega}f\;d\mu_{\Lambda}
\end{array}.
\end{equation*}
\end{definition}

\begin{proposition}
Let $\mu$ be a vector projection family. If $f\in L^{1}(\mu)$, then $\int f\;d\mu\in X^{\ast\ast}$.
\end{proposition}
\begin{proof}
We first assume that $f\in L^{\infty}(\mu)$. The lineality of $\int f\;d\mu$ follows from the lineality of $\Lambda \to \mu_{\Lambda}$, because of which we only need to establish continuity. If $(\Lambda_{n})_{n\in\mathbb{N}}$ is a sequence in $X^{\ast}$ such that $\Lambda_{n} \to \Lambda$ then $\mu_{\Lambda_{n}} \xrightarrow[]{set} \mu_{\Lambda}$. Since $f\in L^{\infty}(\mu)$ this mode of convergence is enough to guarantee that
\begin{equation*}
    \int f\;d\mu_{\Lambda_{n}} \to \int f\;d\mu_{\Lambda}.
\end{equation*}
If $f\in L^{1}(\mu)$ and $f\geq 0$ then there exists a sequence of simple functions $(s_{n})_{n\in\mathbb{N}}$ such that $s_{n} \leq s_{n+1}$ and $s_{n} \xrightarrow[]{pw} f$. This implies that
\begin{equation*}
    \int s_{n}\;d\mu_{\Lambda} \to \int f\;d\mu_{\Lambda}
\end{equation*}
for each $\Lambda\in X^{\ast}$, that is,
\begin{equation*}
    \int s_{n}\;d\mu \xrightarrow[]{pw} \int f\;d\mu.
\end{equation*}
Since each $\int s_{n}\;d\mu$ defines an element of $X^{\ast\ast}$ the Uniform Boundedness Principle implies that $\int f\;d\mu\in X^{\ast\ast}$.
\end{proof}

\begin{corollary}[Monotone Convergence Theorem]
Let $(f_{n})_{n\in\mathbb{N}}$ be a non-decreasing sequence of non-negative functions on $L^{1}(\mu)$ such that $f_{n} \xrightarrow[]{pw} f$. If $f\in L^{1}(\mu)$ then
\begin{equation*}
    \int f_{n}\;d\mu \xrightarrow[]{\omega^{\ast}} \int f\;d\mu.
\end{equation*}
\end{corollary}
\begin{proof}
Given $\Lambda\in X^{\ast}$ we apply the usual form of the Monotone Convergence Theorem to $\mu_{\Lambda}$ to obtain that
\begin{equation*}
    \int f_{n}\;d\mu_{\Lambda} \to \int f\;d\mu_{\Lambda},
\end{equation*}
from where we conclude the result.
\end{proof}

\begin{corollary}[Dominated Convergence Theorem]
If $(f_{n})_{n\in\mathbb{N}}$ is a sequence in $L^{1}(\mu)$ such that $f_{n} \xrightarrow[]{point} f$ and there exists $g\in L^{1}(\mu)$ such that $|f_{n}| \leq g$ for each $n\in\mathbb{N}$ then $f\in L^{1}(\mu)$ and
\begin{equation*}
    \int f_{n}\;d\mu \xrightarrow[]{\omega^{\ast}} \int f\;d\mu.
\end{equation*}
\end{corollary}
\begin{proof}
Let $\Lambda\in X^{\ast}$. By the usual Dominated Convergence Theorem in $L^{1}(\mu_{\Lambda})$ we have that $f\in L^{1}(\mu_{\Lambda})$ and
\begin{equation*}
    \int f\;d\mu_{\Lambda_{n}} \to \int f\;d\mu_{\Lambda},
\end{equation*}
that is,
\begin{equation*}
    \left(\int f_{n}\;d\mu\right)(\Lambda) \to \left(\int f\;d\mu\right)(\Lambda).
\end{equation*}
Since $\Lambda$ is arbitrary we conclude that $\int f_{n}\;d\mu \xrightarrow[]{\omega^{\ast}} \int f\;d\mu$.
\end{proof}

\begin{definition}
Let $X$ be a Banach space, $(\Omega,\Sigma)$ a measurable space, and $\mu$ a vector projection family. We will say that $f\in L^{\infty}(\mu)$ is \textbf{properly integrable} if
\begin{equation*}
    \int_{\Omega} f\;d\mu \in J(X).
\end{equation*}
\end{definition}

There is an immediate corollary to the definition.

\begin{corollary}
Let $X$ be a Banach space, $(\Omega,\Sigma)$ a measurable space, and $\mu$ a vector projection family. If $X$ is reflexive then each element of $L^{\infty}(\mu)$ is properly integrable.
\end{corollary}

\begin{definition}
Let $\mu$ be a vector projection family. Given $A\in\Sigma$ we define the \textbf{semivariation} of $\mu$ in $A$ as
\begin{equation*}
    |\mu|_{SV}(A) = \sup_{\Lambda\in B_{X^{\ast}}} |\mu_{\Lambda}|_{TV}(A).
\end{equation*}
We define the \textbf{semivariation} of $\mu$ as $|\mu|_{SV} = |\mu|_{SV}(\Omega)$.
\end{definition}

We now show that the semivariation of a vector projection family is finita. To this end, we introduce the application
\begin{equation*}
\begin{array}{cccc}
    T\colon& X^{\ast} & \longrightarrow & Med_{\mathbb{C}}(\Omega,\Sigma) \\
    & \Lambda & \longmapsto & \mu_{\Lambda}
\end{array},
\end{equation*}
where $Med_{\mathbb{C}}(\Omega,\Sigma)$ denotes the space of complex measures on $(\Omega,\Sigma)$ with the total variation norm. It is clear from the definition of vector projection family that $T$ is linear. We now show that it is also continuous.

\begin{proposition}
The function $T$ is continuous.
\end{proposition}
\begin{proof}
We appeal to the Closed Graph Theorem. Assume that $\Lambda_{n} \to \Lambda$ and $\mu_{\Lambda_{n}} \xrightarrow[]{TV} \nu$, for some $\nu\in Med_{\mathbb{C}}(\Omega,\Sigma)$. We must show that $\nu = \mu_{\Lambda}$. Since convergence in total variation implies setwise convergence we have that $\mu_{\Lambda_{n}} \xrightarrow[]{set} \nu$. Since $\mu$ is a vector projection family, $\Lambda_{n} \to x$ implies that $\mu_{\Lambda_{n}} \xrightarrow[]{set} \mu_{\Lambda}$. This implies that $\mu_{\Lambda} = \nu$ and establishes the result.
\end{proof}

This result shows that a vector projection family can also be thought of as an operator from $X^{\ast}$ to $Med_{\mathbb{C}}(\Omega,\Sigma)$. More than a simple curiosity, this result implies that the semivariation of a vector projection family is finite.

\begin{corollary}
If $\mu$ is a vector projection family then $\mu$ has finite semivariation.
\end{corollary}
\begin{proof}
The previous lemma shows that $T$ is bounded, hence its norm is finite. To conclude we simply note that
\begin{align*}
    |T| &= \sup_{\Lambda\in B_{X^{\ast}}}|T(\Lambda)|_{TV}\\
    &= \sup_{\Lambda\in B_{X^{\ast}}}|\mu_{\Lambda}|_{TV}\\
    &= |\mu|_{SV}.
\end{align*}
\end{proof}

\subsection{Vector Measures}

\begin{definition}
Let $(\Omega,\Sigma)$ be a measurable space and $X$ a Banach space. A \textbf{vector measure} is a function $\mu\colon \Sigma \to X$ such that
\begin{enumerate}
    \item $\mu(\emptyset) = 0$
    \item $\mu\left(\biguplus_{n\in\mathbb{N}} E_{n}\right) = 
    \sum_{n\in\mathbb{N}}\mu(E_{n})$
\end{enumerate}
\end{definition}

Given a vector measure $\mu$ and $\Lambda\in X^{\ast}$ we can define $\mu_{\Lambda}\colon \Sigma \to \mathbb{C}$ as
\begin{equation*}
    \mu_{\Lambda}(E) = \Lambda(\mu(E)).
\end{equation*}
Furthermore, the application $\Lambda \longmapsto \mu_{\Lambda}$ is linear.

\begin{proposition}
If $\mu$ is a vector measure then the family of measures
\begin{equation*}
    \{\mu_{\Lambda}\;|\;\Lambda\in X^{\ast}\}
\end{equation*}
is a vector projection family.
\end{proposition}
\begin{proof}
We only need to prove continuity. If $(\Lambda_{i})_{i\in I}$ is a net in $X^{\ast}$ such that $\Lambda_{i} \to \Lambda$, in particular it converges pointwise, because of this given $E\in\Sigma$ it follows that
\begin{align*}
    \mu_{\Lambda_{i}}(E) &= \Lambda_{i}(\mu(E)(x))\\
    &\to \Lambda(\mu(E)(x))\\
    &= \mu_{\Lambda}(E)(x).
\end{align*}
\end{proof}

The integral with respect to the vector projection family generated by a vector measure coincides with the generalized Lewis integral as developed in \cite{Lewis}, because of which the integral with respect to a vector projection family generalizes the Lewis integral.

\begin{proposition}
A vector projection family $\mu$ is generated by a vector measure if and only if for each $E\in\Sigma$ the application $\Lambda \longmapsto \mu_{\Lambda}(E)$ is $\tau_{\omega^{\ast}}$ continuous.
\end{proposition}
\begin{proof}
The forward implication is immediate. For the reciprocal implication, our hypothesis implies that the function $\Lambda \longmapsto \mu_{\Lambda}(E)$ is $\tau_{\omega^{\ast}}$ continuous, which implies that there exists $\mu(E)\in X$ such that $\Lambda(\mu(E)) = \mu_{\Lambda}(E)$. It follows that $\Lambda(\mu(\emptyset)) = 0$ for each $\Lambda\in X^{\ast}$ and the Hahn-Banach theorem implies that $\mu(\emptyset) = 0$. On the other hand, if $(E_{n})_{n\in\mathbb{N}}$ is a sequence of disjoint sets in $\Sigma$ then
\begin{align*}
    \Lambda\left(\mu\left(\bigcup_{n\in\mathbb{N}}E_{n}\right)\right) &= \mu_{\Lambda}\left(\bigcup_{n\in\mathbb{N}}E_{n}\right)\\
    &= \sum_{n\in\mathbb{N}}\Lambda(\mu(E_{n})).
\end{align*}
By the Orlicz-Pettis theorem the function $E\longmapsto \mu(E)$ is $\sigma$-additive in the weak topology and therefore is $\sigma$-aditiva.
\end{proof}

To obtain more criteria for a vector projection family to be a vector measure we will use the following lemma.

\begin{lemma}
Let $(\mu_{i})_{i\in I}$ be a net bounded in total variation such that $\mu_{i} \xrightarrow[]{set} \mu$. If $f$ is bounded then
\begin{equation*}
    \int f\;d\mu_{i} \to \int f\;d\mu.
\end{equation*}
\end{lemma}
\begin{proof}
Since $(\mu_{i})_{i\in I}$ is bounded in total variation there exists $M\geq 0$ such that $|\mu_{i}|_{TV} \leq M$ for each $i\in I$. Since $f$ is bounded there exists a sequence of simple functions $(s_{n})_{n\in\mathbb{N}}$ such that $s_{n} \xrightarrow[]{unif} f$. Since each $s_{n}$ is simple and $\mu_{i} \xrightarrow[]{set} \mu$ we have that $\int s_{n}\;d\mu_{i} \xrightarrow[i\to\infty]{} \int s_{n}\;d\mu$ for each $n\in\mathbb{N}$. It follows that
\begin{align*}
    \left| \int f\;d\mu_{i} - \int f\;d\mu \right| &= \left| \int f\;d\mu_{i} - \int s_{n}\;d\mu_{i} \right| + \left| \int s_{n}\;d\mu_{i} - \int s_{n}\;d\mu \right| + \left| \int s_{n}\;d\mu - \int f\;d\mu \right|\\
    &\leq |f - s_{n}|_{\infty}|\mu_{i}|_{TV} + \left| \int s_{n}\;d\mu_{i} - \int s_{n}\;d\mu \right| + |s_{n} - f|_{\infty}|\mu|_{TV}\\
    &\leq |f - s_{n}|_{\infty}M + \left| \int s_{n}\;d\mu_{i} - \int s_{n}\;d\mu \right| + |s_{n} - f|_{\infty}|\mu|_{TV}
\end{align*}
for each $i\in I$ and $n\in\mathbb{N}$. Choose $n\in\mathbb{N}$ such that $|s_{n} - f|_{\infty} < \min\{\frac{\epsilon}{3M},\frac{\epsilon}{3|\mu|_{TV}}\}$ and for this $n$ there exists $i_{0}\in I$ such that if $i\succeq i_{0}$ then $\left| \int s_{n}\;d\mu_{i} - \int s_{n}\;d\mu \right| \leq \frac{\epsilon}{3}$, because of which if $i\succeq i_{0}$ then
\begin{align*}
    \left| \int f\;d\mu_{i} - \int f\;d\mu \right| < \epsilon.
\end{align*}
\end{proof}

\begin{proposition}
A vector projection family is generated by a vector measure if and only if each element of $f\in L^{\infty}(\mu)$ is properly integrable.
\end{proposition}
\begin{proof}
The reverse implication is immediate. For the remaining implication assume that for each $E\in\Sigma$ the function $\Lambda \longmapsto \mu_{\Lambda}(E)$ is $\tau_{\omega^{\ast}}$ continuous, from this if $(\Lambda_{i})_{i\in I}$ is a bounded net in $X^{\ast}$ such that $\Lambda_{i} \xrightarrow[]{\omega^{\ast}} \Lambda$ then $\mu_{\Lambda_{i}} \xrightarrow[]{set} \mu_{\Lambda}$. Since $\mu$ has finite semivariation and $|\mu_{\Lambda_{i}}|_{TV} \leq |\Lambda_{i}||\mu|_{SV} \leq M|\mu|_{SV}$ we have that $(\mu_{\Lambda_{i}})_{i\in I}$ satisfies the hypotheses of the previous lemma. We conclude that $\int f\;d\mu_{\Lambda_{i}} \to \int f\;d\mu_{\Lambda}$. It follows that $\int f\;d\mu$ is continuous with respecto to $\tau_{b\omega^{\ast}}$, because of which it is also continuous with respect to $\tau_{\omega^{\ast}}$ and $\int f\;d\mu\in J(X)$.
\end{proof}

\subsection{Semivariation and Dominated Convergence}

We now prove that the Dominated Convergence Theorem remains valid for properly integrable functions. This will require some work and is based on further properties of the semivariation.

\begin{definition}
Let $\mu$ be a vector projection family. We say that the 
semivariation of $\mu$ is \textbf{continuous} if for every sequence $(E_{n})_{n\in\mathbb{N}}$ in $\Sigma$ such that $E_{n} \searrow \emptyset$ the equality
\begin{equation*}
    \lim_{n\to\infty}|\mu|_{SV}(E_{n}) = 0
\end{equation*}
is satisfied.
\end{definition}

This definition is reminiscent of the continuity of measure and can be summarized by saying that $|\mu|_{SV}$ respects the limit of sequences that decrease to the empty set.

\begin{lemma}
Let $\mu$ be a vector measure. For each $E\in\Sigma$ the inequality
\begin{equation*}
    |\mu|_{SV}(E) \leq 4 \sup_{F\subset E} |\mu(F)|
\end{equation*}
is satisfied.
\end{lemma}
\begin{proof}
By definition of semivariation, we have that
\begin{align*}
    |\mu|_{SV}(E) &= \sup_{\Lambda\in X^{\ast}} |\Lambda\circ\mu|(E)\\
    &\leq \sup_{\Lambda\in X^{\ast}}\Big(|(Re\;\Lambda\circ\mu)^{+}|(E) + |(Re\;\Lambda\circ\mu)^{-}|(E) + |(Im\;\Lambda\circ\mu)^{+}|(E) + |(Im\;\Lambda\circ\mu)^{-}|(E)\Big)\\
    &\leq \sup_{\Lambda\in X^{\ast}}\sup_{F\subset E}|(Re\;\Lambda\circ\mu)^{+}|(F) + \sup_{\Lambda\in X^{\ast}}\sup_{F\subset E}|(Re\;\Lambda\circ\mu)^{-}|(F)\\
    &\quad + \sup_{\Lambda\in X^{\ast}}\sup_{F\subset E}|(Im\;\Lambda\circ\mu)^{+}|(F) + \sup_{\Lambda\in X^{\ast}}\sup_{F\subset E}|(Im\;\Lambda\circ\mu)^{-}|(F)\\
    &\leq 4 \sup_{\Lambda\in X^{\ast}}\sup_{F\subset E} |\Lambda \circ \mu (F)|\\
    &= 4 \sup_{F\subset E}\sup_{\Lambda\in X^{\ast}} |\Lambda \circ \mu (F)|\\
    &= 4 \sup_{F\subset E}|\mu (F)|.
\end{align*}
\end{proof}

\begin{proposition}\label{PropContinuidadSemivariaciónVectorial}
If $\mu$ is a vector measure then its semivariation is continuous.
\end{proposition}
\begin{proof}
Let $(E_{n})_{n\in\mathbb{N}}$ be a sequence of measurable sets that decreases to the empty set. If $|\mu|_{SV}$ is not continuous then there exists $\epsilon > 0$ such that for each $n\in\mathbb{N}$ there exists $\Lambda_{n}\in B_{X^{\ast}}$ such that
\begin{equation*}
    |\mu_{\Lambda_{n}}|(E_{n}) \geq \epsilon.
\end{equation*}
Since $|\mu_{\Lambda_{1}}|$ is a measure we have that $\lim_{n\to\infty} |\mu_{\Lambda_{1}}|(E_{n}) = 0$, because of which there exists $n_{2} > 1$ such that
\begin{equation*}
    |\mu_{\Lambda_{1}}|(E_{n_{2}}) < \frac{\epsilon}{2}.
\end{equation*}
Using this and the previous lemma we obtain that
\begin{align*}
    \frac{\epsilon}{2} &< |\mu_{\Lambda_{1}}|(E_{1}\setminus E_{n_{2}})\\
    &\leq \sup_{\Lambda\in B_{X^{\ast}}} |\mu_{\Lambda}|(E_{1}\setminus E_{n_{2}})\\
    &= |\mu|_{SV}(E_{1}\setminus E_{n_{2}})\\
    &\leq 4 \sup_{F\subset E_{1}\setminus E_{n_{2}}}|\mu(F)|.
\end{align*}
In consequence, there exists $F_{1} \subset E_{1}\setminus E_{n_{2}}$ such that
\begin{equation*}
    |\mu(F_{1})| \geq \frac{\epsilon}{8}.
\end{equation*}
Through an induction argument we obtain an increasing sequence of natural numbers $(n_{k})_{k\in\mathbb{N}}$ and a sequence of measurable sets $(F_{k})_{k\in\mathbb{N}}$ such that $F_{k} \subset E_{n_{k}}\setminus E_{n_{k+1}}$ and
\begin{equation}\label{EqSumaFea}
    |\mu(F_{k})| \geq \frac{\epsilon}{8}.
\end{equation}
Since the sets $F_{k}$ are disjoint and $\mu$ is $\sigma$-additive we have that
\begin{equation*}
    \mu\left(\biguplus_{k\in\mathbb{N}}F_{k}\right) = \sum_{k\in\mathbb{N}} \mu(F_{k}),
\end{equation*}
where the last series is invariant under reordering and therefore absolutely convergent. This contradicts inequality (\ref{EqSumaFea}).
\end{proof}

\begin{lemma}
Let $\mu$ be a vector projection family and $g$ a properly integrable function. The function $\mu^{g}\colon \Sigma \to X$ defined as
\begin{equation*}
    \mu^{g}(E) = \int_{E}g\;d\mu
\end{equation*}
is a vector measure with semivariation
\begin{equation*}
    |\mu^{g}|_{SV}(E) = \sup_{\Lambda\in B_{X^{\ast}}}\int_{E}|g|\;d|\mu_{\Lambda}|.
\end{equation*}
In particular, this last quantity is finite.
\end{lemma}
\begin{proof}
It is enough to verify $\sigma$-additivity. Since
\begin{equation*}
    \Lambda\circ \mu^{g} (E) = \int_{E}f\;d\mu_{\Lambda}
\end{equation*}
is a measure for each $\Lambda\in X^{\ast}$, we have that $\mu^{g}$ is $\sigma$-additive with respect to $\tau_{\omega}$. It follows from the Orlicz-Pettis theorem that $\mu^{g}$ is $\sigma$-additive and thus a vector measure. To compute its semivariation we note that
\begin{align*}
    |\mu^{g}|_{SV}(E) &= \sup_{\Lambda\in B_{X^{\ast}}} |\Lambda\circ \mu^{g}|(E)\\
    &= \sup_{\Lambda\in B_{X^{\ast}}} \int_{E}|g|\;d|\mu_{\Lambda}|.
\end{align*}
\end{proof}

\begin{theorem}[Dominated Convergence for Properly Integrable Functions]
Let $\mu$ be a vector projection family. If $(f_{n})_{n\in\mathbb{N}}$ is a sequence of properly integrable functions such that $f_{n} \xrightarrow[]{pw} f$ and there exists a properly integrable function $g$ such that $|f_{n}| \leq g$ for each $n\in\mathbb{N}$ then $f$ is properly integrable and
\begin{equation*}
    \int f_{n}\;d\mu \to \int f\;d\mu.
\end{equation*}
\end{theorem}
Note that $\mu$ need not have continuous semivariation.
\begin{proof}
We first show that $\int f\;d\mu \in X^{\ast\ast}$ is continuous with respect to $\tau_{b\omega^{\ast}}$. This would imply that it is continuous with respect to $\tau_{\omega^{\ast}}$ and therefore belongs to $J(X)$. To this end we consider a bounded net $(\Lambda_{i})_{i\in I}$ such that $\Lambda_{i} \xrightarrow[]{\omega^{\ast}} \Lambda$. We assume that $(\Lambda_{i})_{i\in I}$ belongs to $B_{X^{\ast}}$. If we define
\begin{equation*}
    E_{n}(\epsilon) = \{x\in\Omega\;|\;|f(x) - f_{n}(x)| \geq \epsilon\}
\end{equation*}
it follows that
\begin{align*}
    \left| \int f\;d\mu_{\Lambda_{i}} - \int f_{n}\;d\mu_{\Lambda_{i}} \right| &\leq \int |f - f_{n}|\;d|\mu_{\Lambda_{i}}|\\
    &\leq \epsilon |\mu|_{SV} + 2 \sup_{\Lambda\in B_{X^{\ast}}}\int_{E_{n}} g\;d|\mu_{\Lambda}|\\
    &\leq \epsilon |\mu|_{SV} + 2 |\mu^{g}|_{SV}(E_{n}).
\end{align*}
Since $\mu^{g}$ is a vector measure, the proposition \ref{PropContinuidadSemivariaciónVectorial} implies that $\mu^{g}$ has continuous semivariation. Since $(E_{n}(\epsilon))_{n\in\mathbb{N}}$ decreases to the empty set $\lim_{n\to\infty} |\mu_{g}|_{SV}(E_{n}) = 0$, thus the right-hand side of the previous estimate can be made arbitrarily small uniformly in $i\in I$. This allows us to conclude that given $\epsilon > 0$ there exists $N_{\epsilon}\in\mathbb{N}$ such that
\begin{equation}\label{EstimadoTCD}
    \left| \int f\;d\mu_{\Lambda_{i}} - \int f_{N_{\epsilon}}\;d\mu_{\Lambda_{i}} \right| < \frac{\epsilon}{3}
\end{equation}
for each $i\in I$. $N_{\epsilon}$ can be made larger, if necessary, to assume that the same happens if $\Lambda_{i}$ is replaced by $\Lambda$. Additionally, since $\int f_{N_{\epsilon}}\;d\mu\in J(X)$ and $\Lambda_{i} \xrightarrow[]{\omega^{\ast}} \Lambda$, there exists $i_{0}\in I$ such that if $i\succeq i_{0}$ then
\begin{equation*}
    \left| \int f_{N_{\epsilon}}\;d\mu_{\Lambda_{i}} - \int f_{N_{\epsilon}}\;d\mu_{\Lambda} \right| < \frac{\epsilon}{3}.
\end{equation*}
We conclude that if $i\succeq i_{0}$ then
\begin{align*}
    \left| \int f\;d\mu_{\Lambda_{i}} - \int f\;d\mu_{\Lambda} \right| &\leq \left| \int f\;d\mu_{\Lambda_{i}} - \int f_{N_{\epsilon}}\;d\mu_{\Lambda_{i}} \right| + \left| \int f_{N_{\epsilon}}\;d\mu_{\Lambda_{i}} - \int f_{N_{\epsilon}}\;d\mu_{\Lambda} \right| + \left| \int f_{N_{\epsilon}}\;d\mu_{\Lambda} - \int f\;d\mu_{\Lambda} \right|\\
    &< \frac{\epsilon}{3} + \frac{\epsilon}{3} + \frac{\epsilon}{3}\\
    &= \epsilon.
\end{align*}
This implies that $\int f\;d\mu \in J(X)$, therefore we will treat it as an element of $X$. To obtain the convergence $\int f_{n}\;d\mu \to \int f\;d\mu$ we note that the estimate (\ref{EstimadoTCD}) extends to each element of $B_{X^{\ast}}$, that is to say that there exists $N\in\mathbb{N}$ such that if $n\geq N$ then
\begin{equation*}
    \left| \int f\;d\mu_{\Lambda} - \int f_{n}\;d\mu_{\Lambda} \right| < \frac{\epsilon}{2}
\end{equation*}
for each $\Lambda\in X^{\ast}$. Since the functions in the estimate are properly integrable we obtain that
\begin{equation*}
    \left| \Lambda\left( \int f\;d\mu - \int f_{n}\;d\mu \right) \right| < \frac{\epsilon}{2}
\end{equation*}
for each $\Lambda\in X^{\ast}$. It follows that
\begin{align*}
    \left| \int f\;d\mu - \int f_{n}\;d\mu \right| &= \sup_{\Lambda\in X^{\ast}} \left| \Lambda\left( \int f\;d\mu - \int f_{n}\;d\mu \right) \right|\\
    &\leq \frac{\epsilon}{2}\\
    &< \epsilon.
\end{align*}
This allows us to conclude the result.
\end{proof}

\section{Operator-Valued Measures}

\begin{definition}
Let $X$ be a Banach space and $(\Omega,\Sigma)$ a measurable space. An \textbf{operator projection family} is a collection of measures in $\Omega$, denoted as
\begin{equation*}
    \mu = \{\mu_{\Lambda,x}\;|\;\Lambda\in X^{\ast},\;x\in X\},
\end{equation*}
with the following properties.
\begin{enumerate}
    \item The function
    \begin{equation*}
    \begin{array}{ccc}
        X^{\ast} \times X & \longrightarrow &\mathcal{M}(\Omega) \\
        (\Lambda,x) & \longmapsto & \mu_{\Lambda,x}
    \end{array}
    \end{equation*}
    defines a bilinear functional.
    \item If $(\Lambda_{i})_{i\in I}$ and $(x_{j})_{j\in J}$ are nets such that $\Lambda_{i} \to \Lambda$ and $x_{j} \to x$ then
    \begin{equation*}
        \mu_{\Lambda_{i},x} \xrightarrow[]{set} \mu_{\Lambda,x}
    \end{equation*}
    and
    \begin{equation*}
        \mu_{\Lambda,x_{j}} \xrightarrow[]{set} \mu_{\Lambda,x}.
    \end{equation*}
\end{enumerate}
\end{definition}

\begin{definition}
Let $X$ be a Banach space, $(\Omega,\Sigma)$ a measurable space, and $\mu$ and operator projection family.
\begin{enumerate}
    \item We will say that $E\in\Sigma$ is a \textbf{null set} if $\mu_{\Lambda,x}(E) = 0$ for each $\Lambda\in X^{\ast}$ and $x\in X$. In this case, we will write $E\in \mathcal{N}(\mu)$.

    \item We will say that a function $f\colon \Omega \to \mathbb{C}$ is \textbf{essentially bounded} with respect to $\mu$ if there exists a null set $E$ such that $f$ is bounded in $\Omega\setminus E$. In this case, we will write $f\in L^{\infty}(\mu)$.

    \item We will say that a measurable function $f\colon \Omega \to \mathbb{C}$ is \textbf{integrable} if $f\in L^{1}(\mu_{\Lambda,x})$ for each $\Lambda\in X^{\ast}$ and $x\in X$. In this case, we will write $f\in L^{1}(\mu)$.
\end{enumerate}
\end{definition}

We note that since each $\mu_{\Lambda,x}$ is finite we also have that $L^{\infty}(\mu) \subset L^{1}(\mu)$.

\begin{proposition}
Let $\mu$ be an operator projection family. For each $\Lambda\in X^{\ast}$ and $x\in X$ the families
\begin{equation*}
    \mu(x) = \{\mu_{\Lambda,x}\;|\;\Lambda\in X^{\ast}\}
\end{equation*}
and
\begin{equation*}
    \Lambda(\mu) = \{\mu_{\Lambda,x}\;|\;x\in X\}
\end{equation*}
are vector projection families.
\end{proposition}
\begin{proof}
The bilinearity and separate continuity of $\mu$ imply the lineality and continuity of $\mu(x)$ and $\Lambda(\mu)$.
\end{proof}

The immediate corollary is the following.

\begin{corollary}
Let $\mu$ be an operator projection family and $f\in L^{1}(\mu)$. The function
\begin{equation*}
\begin{array}{cccc}
    \int_{\Omega}f\;d\mu \colon& X \times X^{\ast\ast} & \longrightarrow & \mathbb{C}\\
    & (\Lambda,x) & \longmapsto & \int_{\Omega}f\;d\mu_{\Lambda,x}
\end{array}
\end{equation*}
is bilinear and continuous.
\end{corollary}
\begin{proof}
Bilineality follows from the bilineality of $(\Lambda,x) \longmapsto \mu_{\Lambda,x}$. Since we are working in Banach spaces, continuity is equivalent to separate continuity in each of the variables. The separate continuity in each variable follows from the corresponding result for vector projection families applied to the families $\Lambda(\mu)$ and $\mu(x)$.
\end{proof}

Given an operator projection family $\mu$ and $f\in L^{1}(\mu)$, the function
\begin{equation*}
\begin{array}{cccc}
    \int_{\Omega}f\;d\mu(x)\colon& X^{\ast} & \longrightarrow &\mathbb{C} \\
    & \Lambda & \longmapsto & \int_{\Omega}f\;d\mu_{\Lambda,x}
\end{array}
\end{equation*}
defines an element of $X^{\ast\ast}$, therefore the application
\begin{equation*}
\begin{array}{cccc}
    \int_{\Omega}f\;d\mu \colon& X & \longrightarrow &X^{\ast\ast}\\
    & x & \longmapsto & \int_{\Omega}f\;d\mu(x)
\end{array}
\end{equation*}
defines an element of $B(X,X^{\ast\ast})$. We define the integral of $f$ with respect to $\mu$ to be this application. Note that this operator is the corresponding one to the bilinear form considered in the previous corollary, that is, the bounded bilinear functional
\begin{equation*}
\begin{array}{cccc}
    \int_{\Omega}f\;d\mu \colon& X \times X^{\ast\ast} & \longrightarrow & \mathbb{C}\\
    & (\Lambda,x) & \longmapsto & \int_{\Omega}f\;d\mu_{\Lambda,x}
\end{array}.
\end{equation*}
Analogously, the function
\begin{equation*}
\begin{array}{cccc}
    \int_{\Omega}f\;d\Lambda(\mu)\colon& X & \longrightarrow &\mathbb{C} \\
    & x & \longmapsto & \int_{\Omega}f\;d\mu_{\Lambda,x}
\end{array}
\end{equation*}
defines an element of $X^{\ast}$, therefore the application
\begin{equation*}
\begin{array}{cccc}
    \int_{\Omega}f\;d\mu \colon& X^{\ast} & \longrightarrow &X^{\ast}\\
    & \Lambda & \longmapsto & \int_{\Omega}f\;d\Lambda(\mu)
\end{array}
\end{equation*}
defines an element of $B(X^{\ast})$. Any of these functions is an acceptable form of the integral, as they are all compatible with each other. However, it will be more natural to consider the integral as an element of $B(X,X^{\ast\ast})$.

\begin{definition}
Let $X$ be a Banach space, $(\Omega,\Sigma)$ a measurable space, and $\mu$ an operator projection family. We will say that $f\in L^{1}(\mu)$ is \textbf{properly integrable} if
\begin{equation*}
    \int_{\Omega} f\;d\mu (X) \subset J(X).
\end{equation*}
\end{definition}

Once again, there is an immediate corollary of the definition.

\begin{corollary}
Let $X$ be a Banach space, $(\Omega,\Sigma)$ a measurable space, and $\mu$ an operator projection family. If $X$ is reflexive then each element of $L^{1}(\mu)$ is properly integrable.
\end{corollary}

The simple version of the Dominated Convergence theorem is also immediately translated to the operator case.

\begin{theorem}[Dominated Convergence Theorem]
Let $\mu$ be an operator projection family. If $(f_{n})_{n\in\mathbb{N}}$ is a sequence of functions in $L^{1}(\mu)$ such that $f_{n} \xrightarrow[]{point} f$ and there exists $g\in L^{1}(\mu)$ such that $|f_{n}| \leq g$ for each $n\in\mathbb{N}$ then $f\in L^{1}(\mu)$ and
\begin{equation*}
    \int f_{n}\;d\mu(x) \xrightarrow[]{\omega^{\ast}} \int f\;d\mu(x).
\end{equation*}
\end{theorem}
\begin{proof}
Apply the corresponding result of vector projection families to $\mu(x)$ for each $x\in X$.
\end{proof}

\begin{theorem}[Monotone Convergence Theorem]
Let $(f_{n})_{n\in\mathbb{N}}$ be a non-decreasing sequence of non-negative functions in $L^{1}(\mu)$ such that $f_{n} \xrightarrow[]{pw} f$. If $f\in L^{1}(\mu)$ then
\begin{equation*}
    \int f_{n}\;d\mu(x) \xrightarrow[]{\omega^{\ast}} \int f\;d\mu(x).
\end{equation*}
\end{theorem}
\begin{proof}
Apply the corresponding result of vector projection families to $\mu(x)$ for each $x\in X$.
\end{proof}

We now give the definition that motivated our study.

\begin{definition}
Let $(X,|\cdot|)$ be a Banach space and $(\Omega,\Sigma)$ a measurable space. An \textbf{operator measure} (or operator-valued measure) is a function $\mu\colon \Sigma \to B(X)$ with the following properties
\begin{enumerate}
    \item $\mu(\emptyset) = 0$ and $\mu(\Omega) = Id$
    \item $\mu(E_{1}\uplus E_{2}) = \mu(E_{1}) + \mu(E_{2})$
    \item For each $x\in X$ and $\Lambda\in X^{\ast}$ the function
    \begin{equation*}
    \begin{array}{cccc}
        \mu_{\Lambda,x}\colon & \Sigma & \longrightarrow & \mathbb{F}\\
        & E & \longmapsto & \Lambda(\mu(E)(x))
    \end{array}
    \end{equation*}
    is a measure.
\end{enumerate}
\end{definition}

The third property of the definition shows that there is a natural way of obtaining an operator projection family from an operator measure.

\begin{proposition}
Let $\mu$ be an operator measure. The family of measures
\begin{equation*}
    \mu_{\Lambda,x}(E) = \Lambda(\mu(E)(x))
\end{equation*}
is an operator projection family.
\end{proposition}
\begin{proof}
Since $\mu(E)\in B(X)$ for each $x\in X$, we have that the application $(\Lambda,x) \longmapsto \mu_{\Lambda,x}$ is immediately bilinear and continuous.
\end{proof}

\begin{proposition}
Let $\mu$ be an operator projection family. $\mu$ is generated by an operator measure if and only if for each $E\in\Sigma$ the application
\begin{equation*}
    (\Lambda,x) \longmapsto \mu_{\Lambda,x}(E)
\end{equation*}
is continuous with respect to $\tau_{\omega^{\ast}}$ in the first variable and continuous with respect to $\tau_{\omega}$ in the second variable.
\end{proposition}
\begin{proof}
The necessity of these conditions is immediate. For sufficiency we note that the corresponding result for vector projection families applied to $\mu(x)$ provides the existence of vectors $\mu(E)(x)$ such that
\begin{equation*}
    \Lambda(\mu(E)(x)) = \mu_{\Lambda,x}(E)
\end{equation*}
for each $\Lambda\in X^{\ast}$. It follows that if $(x_{i})_{i\in I}$ is a net in $X$ such that $x_{i}\xrightarrow[]{\omega} x$ then
\begin{equation*}
    \Lambda(\mu(E)(x_{i})) \to \Lambda(\mu(E)(x))
\end{equation*}
for each $\Lambda\in X^{\ast}$, that is,
\begin{equation*}
    \mu(E)(x_{i}) \xrightarrow[]{\omega} \mu(E)(x).
\end{equation*}
This implies that the linear map $x\longmapsto \mu(E)(x)$ is continuous from $(X,\tau_{\omega})$ to $(X,\tau_{\omega})$ and therefore is norm continuous, that is, $\mu(E)\in B(X)$. From this, we conclude that the application $E \longmapsto \mu(E)$ is an operator measure.
\end{proof}

\begin{corollary}
An operator projection family $\mu$ is generated by an operator measure if and only if each element of $L^{\infty}(\mu)$ is properly integrable.
\end{corollary}
\begin{proof}
If $f\in L^{\infty}(\mu)$ then $f\in L^{\infty}(\mu(x))$ for each $x\in X$. Since each $\mu(x)$ is a vector projection family, the corresponding result to this case implies that $\int f\;d\mu(x) \in J(X)$ and therefore
\begin{equation*}
    \int_{\Omega} f\;d\mu (X) \subset J(X).
\end{equation*}
\end{proof}

\begin{theorem}[Dominated Convergence Theorem for Properly Integrable Functions]
Let $\mu$ be an operator projection family. If $(f_{n})_{n\in\mathbb{N}}$ is a sequence of properly integrable functions such that $f_{n} \xrightarrow[]{pw} f$ and there exists a properly integrable function $g$ such that $|f_{n}| \leq g$ for each $n\in\mathbb{N}$ then $f$ is properly integrable and
\begin{equation*}
    \int f_{n}\;d\mu(x) \to \int f\;d\mu(x).
\end{equation*}
In consequence,
\begin{equation*}
    \int f_{n}\;d\mu \xrightarrow[]{SO} \int f\;d\mu.
\end{equation*}
\end{theorem}
\begin{proof}
It follows from the pointwise application of the vector case.
\end{proof}

\section{Examples}

\subsection{Measures in Banach spaces}
Let $X$ be a Banach space. The simplest kind of vector measures are constructed by taking a scalar measure $\lambda$ and a vector $x\in X$ and defining
\begin{equation*}
    \mu(E) = \lambda(E) x.
\end{equation*}
It is trivial that this is a vector measure and can be easily generalized to a finite number of scalar measures $\{\lambda_{n}\}_{n=1}^{m}$ and a finite number of vectors $\{x_{n}\}_{n=1}^{m}$, by defining
\begin{equation*}
    \mu(E) = \sum_{n=1}^{m}\lambda_{n}(E)x_{n}.
\end{equation*}
It is once again trivial that this defines a vector measure. The next generalization, and the first non-trivial one, would be the countable case. Let $(\lambda_{n})_{n\in\mathbb{N}}$ be a sequence of probability measures in a common measurable space $(\Omega,\Sigma)$ and $(x_{n})_{n\in\mathbb{N}}$ a sequence of unit-norm vectors in $X$. Define $\mu\colon \Sigma \to X$ by
\begin{equation*}
    \mu(E) = \sum_{n=1}^{\infty}2^{-n}\lambda_{n}(E)x_{n}.
\end{equation*}
The series converges since $X$ is a Banach space and it clearly converges in norm. By the Orlicz-Pettis Theorem, to show $\sigma$-additivity it is enough to show $\sigma$-additivity in the weak topology. Consider $\Lambda\in X^{\ast}$ and define
\begin{equation*}
    \mu_{\Lambda}(E) = \sum_{n=1}^{\infty}2^{-n}\lambda_{n}(E)\Lambda(x_{n}).
\end{equation*}
The Nikodym Convergence Theorem implies that the function $\mu_{\Lambda}$ is a measure. Therefore, $\mu$ is $\sigma$-additive and a vector measure. This also shows that the collection of measures
\begin{equation*}
    \{\mu_{\Lambda}\;|\;\Lambda\in X^{\ast}\}
\end{equation*}
is a vector projection family. Now consider an essentially bounded function $f$ and consider the vector
\begin{equation*}
    \sum_{n=1}^{\infty}2^{-n}\int f\;d\lambda_{n}\;x_{n},
\end{equation*}
which converges since it converges in norm. For any $\Lambda\in X^{\ast}$ we have that
\begin{equation*}
    \Lambda\left(\sum_{n=1}^{\infty}2^{-n}\int f\;d\lambda_{n}\;x_{n}\right) = \sum_{n=1}^{\infty}2^{-n}\int f\;d\lambda_{n}\;\Lambda(x_{n}).
\end{equation*}
This shows that
\begin{equation*}
    \int f\;d\mu = \sum_{n=1}^{\infty}2^{-n}\int f\;d\lambda_{n}\;x_{n}.
\end{equation*}
Therefore, this first non-trivial example can be completely solved. This same procedure can be generalized to operator-valued measures by considering a sequence $(T_{n})_{n\in\mathbb{N}}$ in $B(X)$, each with unit norm. Define
\begin{equation*}
    \mu(E) = \sum_{n=1}^{\infty}2^{-n}\lambda_{n}(E)T_{n}
\end{equation*}
and
\begin{equation*}
    \mu_{\Lambda,x}(E) = \sum_{n=1}^{\infty}2^{-n}\lambda_{n}(E) \Lambda(T_{n}(x))
\end{equation*}
for each $\Lambda\in X^{\ast}$ and $x\in X$. Each $\mu_{\Lambda,x}$ is a measure by the Nikodym Convergence Theorem and therefore $\mu$ is an operator measure. In this case, the integral of an essentially bounded function $f$ is given by
\begin{equation*}
    \int f\;d\mu(x) = \sum_{n=1}^{\infty}2^{-n}\int f\;d\lambda_{n}\; T_{n}(x).
\end{equation*}
This provides a non-trivial example of an operator-valued measure.

\subsection{Measures in Hilbert Spaces}\label{EjHilbert}
Let $H$ be a Hilbert space, $(\Omega,\Sigma)$ a measurable space, and $\mu$ an operator measure. By the Riesz Representation Theorem for Hilbert spaces, we have that each element of $H^{\ast}$ is of the form
\begin{equation*}
    x \longmapsto \langle x,y\rangle,
\end{equation*}
for a certain $y\in H$. In consequence, the previously studied measures $\mu_{\Lambda,x}$ are characterized by two elements of $H$, thus we have measures
\begin{equation*}
    \mu_{x,y}(A) = \langle \mu(A)x,y\rangle.
\end{equation*}
If $f\colon \Omega\subset\mathbb{C} \to \mathbb{C}$ is integrable with respect to $\mu$ then $\int f\;d\mu \in B(H)$ since $H$ is reflexive and therefore each integrable function is properly integrable. This operator is characterized by satisfying the equation
\begin{align*}
    \left\langle\int f\;d\mu(x),y\right\rangle &= \int f\;d\mu_{x,y}\\
    &= \int f\;d\langle \mu(\cdot)x,y\rangle
\end{align*}
for each $x,y\in H$. If $\{e_{i}\}_{i\in I}$ is a Hilbert basis, with $I$ not necessarily countable, then the integral $\int f\;d\mu$ can be reconstructed through the equation
\begin{equation*}
    \int f\;d\mu(x) = \sum_{i\in I}\left(\int f\;d\mu_{x,e_{i}}\right) e_{i}.
\end{equation*}
We conclude that if $f$ is integrable and $\{e_{i}\}_{i\in I}$ is a Hilbert basis then
\begin{equation*}
    \int f\;d\mu_{x,e_{i}} = 0
\end{equation*}
except for a countable family of indexes and, if $I_{f} = \{i_{n}\;|\;\mathbb{N}\}$ is such set, then the sequence $\left(\int f\;d\mu_{x,e_{i_{n}}}\right)_{n\in\mathbb{N}}$ defines an element of $\ell^{2}(\mathbb{C})$. In this way, the operator measure $\mu$ determines a set of rules for the projections of a vector, integrable functions being the ones that give a set of projections realizable as an element of $H$.

\subsection{Spectral Measures}
Let $H$ be a Hilbert space, $\Omega\subset\mathbb{C}$ a measurable space and $\Sigma_{\Omega}$ the Borel $\sigma$-algebra of $\Omega$. A \textbf{spectral measure} or \textbf{resolution of the identity} is an operator measure $E\colon \Sigma_{\Omega} \to B(H)$ with the following properties:
\begin{enumerate}
    \item for each $A\in\Sigma_{\Omega}$ the operator $E(A)$ is a self-adjoint projection.
    \item $E(A\cap B) = E(A)\circ E(B)$ for each $A,B\in\Sigma_{\Omega}$. In particular, $E(A)$ and $E(B)$ commute.
\end{enumerate}
Let us denote by $\mathcal{M}_{\sigma}(H)$ the set of spectral measures in $H$. All properties of the operator measures discussed in the example \ref{EjHilbert} are also valid for spectral measures. A notable consequence of the second property of spectral measures is that the obtained integral is multiplicative, in the sense that if $f$ and $g$ are integrable then
\begin{equation*}
    \left(\int f\;dE\right) \circ \left(\int g\;dE\right) = \int fg\;dE.
\end{equation*}

The name of these operator measures originates from the spectral theorem.

\begin{theorem}[Spectral Theorem]
Let $H$ be a Hilbert space. For each normal $T\in B(H)$ there exists a unique spectral measure $E^{T}$ defined on the Borel subsets of the spectrum of $T$, $\sigma(T)$, such that
\begin{align*}
    T &= \int_{\sigma(T)} Id\;dE^{T}\\
    &= \int_{\sigma(T)} \lambda\;dE^{T}(\lambda).
\end{align*}
\end{theorem}

With the terminology of our general construction, we can restate the spectral theorem as follows.

\begin{theorem}[Spectral Theorem]
Let $H$ be a Hilbert space. The application
\begin{equation*}
\begin{array}{ccc}
    \mathcal{M}_{\sigma}(H) & \longrightarrow &B(H) \\
    E & \longmapsto & \int_{\sigma(T)}Id\;dE
\end{array}
\end{equation*}
is an isomorphism between spectral measures and the normal elements of $B(H)$.
\end{theorem}

One of the most important consequences of the Spectral Theorem is that it provides the definition of the Borel Functional Calculus; given an integrable function $f\colon \sigma(T) \to \mathbb{C}$ we define $f(T)\in B(H)$ as
\begin{equation*}
    f(T) = \int_{\sigma(T)}f\;dE^{T}.
\end{equation*}
The map $f\longmapsto f(T)$ defines a functional calculus that generalizes the Holomorphic Functional Calculus, in the sense that it extends the inverse of the Gelfand transform.

\subsection{Operation and Positive Operator Measures}
Let $H$ be a Hilbert space. We say that $\rho\in B(H)$ is a \textbf{state} if $\rho$ is trace class and non-negative, that is, for each $x\in H$ we have that $\langle \rho(x),x \rangle \geq 0$. Since non-negative operators are self-adjoint we also have that every state is self-adjoint. We denote the set of states of $H$ by $S(H)$. If $L^{1}(H)$ is the set of trace class operators in $H$ then $S(H)$ is a closed subspace of $(L^{1}(H),|\cdot|_{tr})$, thus $(S(H),|\cdot|_{tr})$ is a Banach space. If $B_{SA}(H)$ is the set of self-adjoint elements of $B(H)$ and $K_{SA}(H)$ is the set of compact self-adjoint operators then we have the following dualities:
\begin{equation*}
    K_{SA}(H) \xrightarrow[]{\ast} S(H) \xrightarrow[]{\ast} B_{SA}(H),
\end{equation*}
by which we mean that $K_{SA}(H)^{\ast}$ is identified with $S(H)$ and $S(H)^{\ast}$ is identified with $B_{SA}(H)$. Specifically, for each $\Lambda\in K_{SA}(H)^{\ast}$ there exists a unique $\rho\in S(H)$ such that
\begin{align*}
    \Lambda(K) &= tr(K\circ\rho)\\
    &= tr(\rho\circ K)
\end{align*}
and for each $\Phi\in S(H)^{\ast}$ there exists an unique $T\in B_{SA}(H)$ such that
\begin{align*}
    \Phi(\rho) &= tr(T\circ\rho)\\
    &= tr(\rho\circ T).    
\end{align*}
The reason we consider these spaces is that we won't consider operator measures in $H$ but rather on the Banach space $S(H)$.

\begin{definition}
Let $H$ be a Hilbert space and $(\Omega,\Sigma)$ a measurable space. An \textbf{operation measure} is an operator measure $\mathcal{E}\colon \Sigma \to B(S(H))$.
\end{definition}

The name of these measures originates from the fact that the elements of $B(S(H))$ are called operations. If $\mathcal{E}$ is an operation measure then the associated measures are of the form
\begin{align*}
    \mathcal{E}_{T,\rho}(A) &= tr(T\circ \mathcal{E}(A)(\rho))\\
    &= tr((\mathcal{E}(A)(\rho))\circ T),
\end{align*}
where $\rho\in S(H)$ and $T\in B(H)$. It follows that if $f$ is a function integrable with respect to $\mathcal{E}$ then the operation $\int f\;d\mathcal{E}$ is characterized by the equation
\begin{align*}
    tr\left(T\circ \int f\;d\mathcal{E}\right) &= \int f\;d\mathcal{E}_{T,\rho}\\
    &= \int f\;d(tr(T\circ \mathcal{E}(\cdot)(\rho))).
\end{align*}
Since $S(H)$ is not in general reflexive it is important to distinguish between integrable and properly integrable functions. If $x\in H$ then the operator $\rho_{x}$ defined as
\begin{equation*}
    \rho_{x}(y) = \langle y,x \rangle x
\end{equation*}
satisfies that $\rho_{x}\in S(H)$, thus the map
\begin{equation}\label{InyecciónHilbertEstados}
\begin{array}{cccc}
    \mathcal{D}\colon & H & \longrightarrow &S(H) \\
     &x  &\longmapsto  &\rho_{x}
\end{array}
\end{equation}
is an injection of $H$ into $S(H)$. The function $\mathcal{D}$ is not linear but it satisfies $\rho_{x+y} = \rho_{x} + \rho_{y}$ if $\langle x,y \rangle = 0$. In this way, it makes sense to ask whether the operation measure $\mathcal{E}$ can be thought of as an operator measure in $H$.

\begin{definition}
Let $H$ be a Hilbert space. A \textbf{positive operator measure} is an operator measure $\mathcal{P}\colon \Sigma \to B(H)$ such that $\mathcal{P}(A)$ is a non-negative operator for each $A\in\Sigma$.
\end{definition}

Positive operator measures are similar to the general operator measures in Hilbert spaces described in the example \ref{EjHilbert}, with the notable difference that if $f$ is a non-negative function then $\int f\;d\mathcal{P}$ is a non-negative operator. This follows from the fact that $\mathcal{P}_{x,x}$ is a non-negative measure, since
\begin{equation*}
    \mathcal{P}_{x,x}(A) = \langle \mathcal{P}(A)(x),x \rangle \geq 0
\end{equation*}
as $\mathcal{P}(A)$ is a non-negative operator. This implies that
\begin{align*}
    \left\langle \int f\;d\mathcal{P}(x),x \right\rangle = \int f\;d\mathcal{P}_{x,x} \geq 0.
\end{align*}
Both kinds of measures are related to Quantum Mechanics and how the state of a quantum system changes when a measurement takes place. A natural problem that arises from this is whether certain kinds of measurements in pure states make sense in mixed states.

\begin{definition}
Let $H$ be a Hilbert space. We say that a positive operator measure $\mathcal{P}$ has an \textbf{extension to mixed states} with respect to a Hilbert $\{e_{i}\}_{i\in I}$ basis if there exists an operation measure $\mathcal{E}$ such that
\begin{equation*}
    \mathcal{E}(A)(\rho_{e_{i}}) = \mathcal{P}(A)(e_{i})
\end{equation*}
for each $i\in I$ and $A\in\Sigma$.
\end{definition}

The reason to ask the previous equation to be valid only for elements of the Hilbert basis is that the function $\mathcal{D}$ defined in (\ref{InyecciónHilbertEstados}) is not linear unless the summands are orthogonal. The problem of whether a positive operator measure has an extension to mixed states is in general very difficult. A way to see this is to note that since $H$ is reflexive each integrable function is properly integrable, while this is not necessarily true for operation measures.

\section{Conclusions}

In this work, we introduced the concept of projection families, which generalize the usual notions of vector and operator-valued measures and still satisfy the theorems of monotone and dominated convergence. Instead of the standard notion of a measure that associates an operator to a measurable set, we use a family of measures that determines the projections of an operator. 

The new notion of projection families and their properties allows us to generalize Lewis integration theory in such a way that previous integration theories with respect to operator-valued measures turn out to be contained as particular cases in this new generalized theory.

As an additional result, it can be shown that projection families allow us to generalize the Spectral Theorem to include Banach algebras and operators between Banach spaces. 
Indeed, the classical Spectral Theorem is a result on operators acting on Hilbert spaces that may fail in Banach spaces for various reasons. This theorem states that a normal operator acting on a Hilbert space can be written as an integral with respect to an operator-valued measure. Our generalization states that the previous statement can be made valid in Banach spaces if the notion of 
operator-valued measure is replaced by  that of operator projection families. The main difficulty is defining a continuous functional calculus that does not rely on the Gelfand Theory of commutative algebras.
This result is not straightforward and requires detailed explanations, which will be presented in a follow-up article. 

\section*{Acknowledgements}
This work was supported by UNAM-DGAPA-PAPIIT, grant No. IN108225, and CONAHCYT, grant No. CBF-2025-I-243.\\




\nocite{*}
\bibliographystyle{alpha}
\bibliography{main}
\addcontentsline{toc}{section}{Bibliography}

\end{document}